\documentclass[11pt]{article}

\usepackage[a4paper,margin=1.15in]{geometry}
\usepackage{amsmath,amssymb,amsthm,mathtools}
\usepackage[numbers,sort&compress]{natbib}
\usepackage[colorlinks=true,citecolor=blue,linkcolor=blue,urlcolor=blue]{hyperref}

\usepackage{graphics}
\usepackage{tikz}
\usetikzlibrary{patterns, patterns.meta}

\newtheorem{theorem}{Theorem}
\newtheorem{conjecture}[theorem]{Conjecture}

\newtheorem{claim}[theorem]{Claim}

\theoremstyle{definition}
\newtheorem{remark}[theorem]{Remark}
\newtheorem{construction}{Construction}

\usepackage[T1]{fontenc}
\DeclareMathOperator{\Maj}{Maj}
\newcommand{\Majp}{\Maj'}

\title{Strong Majority Edge-Coloring}
\author{
Sylwia Antoniuk\thanks{
Faculty of Mathematics and Computer Science, Adam Mickiewicz University, Pozna\'n, Poland.
E-mail: \texttt{antoniuk@amu.edu.pl}.
}
\and Magdalena Prorok\thanks{
AGH University of Krakow, al.~Mickiewicza 30, 30-059 Kraków, Poland. 
E-mail: \texttt{prorok@agh.edu.pl}.
}
\and Nika Salia\thanks{
Theoretical Computer Science Department, Faculty of Mathematics and Computer Science, 
Jagiellonian University, Kraków, Poland. 
E-mail: \texttt{nika.salia@uj.edu.pl}.
}
}   
\date{}

\begin{document}

\maketitle

\begin{abstract}
A strong majority edge-coloring of a graph is an edge-coloring in which, for
every edge \(e\) and every color \(i\), at most half of the edges adjacent to \(e\) have color \(i\). Such a coloring exists only for graphs with no pendant path of length two, which, following Kalinowski, Kamyczura, Pilśniak,
and Woźniak, we call admissible. They proved that every admissible graph admits such a coloring with at most eight colors and conjectured that four colors always suffice. We improve the upper bound from eight to five.
\end{abstract}

\section{Introduction}

Majority colorings form a natural relaxation of proper colorings. Instead of forbidding local coincidences altogether, one asks only that no color dominates a local neighborhood. In its classical vertex form, a coloring is a majority coloring if each vertex has at most half of its neighbors in its own color.  Variants of this idea have been studied for graphs, digraphs, and edge-colorings.

Majority colorings go back to Lovász's decomposition theorem, which implies that every finite graph has a majority vertex-coloring with two colors~\cite{lovasz1966decomposition}.
The infinite version is closely related to the unfriendly partition problem, for which several important positive results are known~\cite{aharoni1990unfriendly,bruhn2010rayless,berger2017unfriendly,kalinowski2025unfriendly}.
Directed majority colorings were introduced by Kreutzer, Oum, Seymour, van der Zypen,
and Wood, who proved the four-color theorem and conjectured that three colors suffice.
On the edge-coloring side, Bock, Kalinowski, Pardey, Pilśniak, Rautenbach, and Woźniak~\cite{bock2023majority}
proved that four colors suffice for the majority edge-colorings, while Hilton--de Werra-type
equitable coloring tools and vertex-splitting arguments provide useful technical
background~\cite{hilton1994sufficient,hilton1998vertex}.

Strong majority colorings were recently introduced by Kalinowski, Kamyczura, Pilśniak, 
and Woźniak~\cite{KKPW}. Here, the adjective ``strong'' refers to the fact that the restriction is imposed simultaneously for every color, not only for the color of the vertex or edge under consideration. A \emph{strong majority vertex-coloring} of a graph 
$G$ is a coloring $c:V(G)\to C$ in which for every vertex $v\in V(G)$ and every color $i\in C$, at most half of the neighbors of $v$ have color $i$. The least number of colors in such a coloring is denoted by \(\Maj(G)\). 
In the edge version, an edge-coloring $c:E(G)\to C$ is a \emph{strong majority edge-coloring} if, for every edge \(uv\in E(G)\) and every color \(i\in C\),
\[
        \bigl|\{f\in E(G): |\{u,v\}\cap f|=1,~c(f)=i\}\bigr|
        \le
        \frac{d(u)+d(v)}{2}-1.
\]
In other words, for every edge $e\in E(G)$ and every color $i\in C$, at most half of the edges adjacent to $e$ have color $i$. The least number of colors in such a coloring is denoted by \(\Majp(G)\).

A strong majority edge-coloring need not always exist. Indeed, there is a single trivial obstruction: if a graph $G$ contains a path \(xyz\) with \(d(x)=1\) and \(d(y)=2\), then the edge \(xy\) has exactly one adjacent edge, namely \(yz\), and no color can appear on at most the half of the edge neighborhood of \(xy\). We therefore restrict our attention to \emph{admissible graphs}, that is, graphs with no pendant path of length two.  
Conversely, every admissible graph admits a strong majority edge-coloring if sufficiently many colors are allowed, for instance, assigning distinct colors to all edges is an example of such a coloring.

Kalinowski, Kamyczura, Pilśniak, and Woźniak \cite{KKPW} proved that $\Maj(G)$ can be arbitrarily large, but no larger than $2\Delta(G)+1$ for graphs with no vertices of degree one.
On the other hand, they also proved that \(\Majp(G)\le 8\) for
every admissible graph \(G\), and that \(\Majp(G)\le 4\) whenever the minimum
degree satisfies \(\delta(G)\ge 7\). They also conjectured that \(4\) colors always
suffice.

\begin{conjecture}[\cite{KKPW}]
Every admissible graph \(G\) satisfies
\[
        \Majp(G)\le 4.
\]
\end{conjecture}

We improve the general bound from \(8\) to \(5\). We further prove that $\Majp(G)\le 4$
whenever \(G\) has no vertices of degree \(2\) or \(4\). In particular, this holds for every graph with \(\delta(G)\ge 5\), lowering the previously known minimum-degree threshold for \(\Majp(G)\le 4\) from \(7\) to \(5\).
\begin{theorem}\label{thm:five}
Every admissible graph \(G\) satisfies
\[
        \Majp(G)\le 5.
\]
\end{theorem}

The main tool in the proof is the equitable edge-coloring theorem of Hilton and
de Werra~\cite{hilton1994sufficient}, stated below. Roughly speaking, an equitable coloring distributes the colors so evenly around each vertex that the strong majority condition is automatic at every edge joining two vertices of degree at least three; the only edges left to handle are those on paths consisting of degree two vertices, and these are adjusted by a local recoloring that does not disturb the balance already secured at the higher-degree ends.

\begin{theorem}[Hilton, de Werra~{\cite{hilton1994sufficient}}]\label{thm:HdW}
Let \(H\) be a simple graph and let \(k\ge 2\). If \(k\nmid d_H(v)\) for every
\(v\in V(H)\), then \(H\) has an equitable edge-coloring with \(k\) colors; that
is, an edge-coloring such that, for every vertex \(v\) and every two colors
\(i,j\in[k] := \{1,2,\dots,k\}\),
\[
        |d_i(v)-d_j(v)|\le 1,
\]
where \(d_i(v)\) denotes the number of edges of color \(i\) incident
with \(v\).
\end{theorem}
The proof of Theorem~\ref{thm:five}, given in Section~\ref{sec:five}, has three stages. We first prepare the graph to satisfy the assumptions of Theorem~\ref{thm:HdW} and make some additional reductions to be able to handle the edges incident with degree two vertices. We then apply Theorem~\ref{thm:HdW} to the reduced graph, to obtain an edge-coloring in which each color appears nearly equally often at every vertex; this already makes the strong majority condition hold at every edge whose two endpoints have degree at least three. The edges that remain to be handled are those incident with at least one vertex of degree two. Their colors are adjusted in the next stage by recoloring along the paths of degree-two vertices, simultaneously preserving the balance at their higher-degree ends. When the graph has no degree-two vertices, the latter stage is vacuous, the equitable coloring is already strong majority on its own, and a sharper count shows that four colors suffice.

\begin{theorem}\label{cor:delta5}
Every graph \(G\) with no vertices of degree \(2\) or \(4\) satisfies
\(\Majp(G)\le 4\).
\end{theorem}

\begin{proof}
 Attach one auxiliary leaf at each vertex whose degree is divisible by four and apply Theorem~\ref{thm:HdW} with \(k=4\) to obtain an equitable coloring with colors in $[4]$. After deleting the auxiliary leaves, every vertex \(v\in V(G)\) satisfies
\[d_i(v)\le\left\lfloor \frac{d(v)}4\right\rfloor+1\le \frac{d(v)-1}2\] 
for each color \(i\in[4]\), \ \(d(v)\notin \{2,4\}\). Hence, for every edge \(uv\in E(G)\) and any color \(i\in[4]\), the number of \(i\)-colored edges adjacent to \(uv\) is at most \(d_i(u)+d_i(v)\le\frac{d(u)+d(v)}{2}-1\), so the coloring is strong majority.
\end{proof}

We conclude this section with a conjecture for Eulerian graphs.

\begin{conjecture}\label{conj:eulerian}
Every Eulerian graph \(G\) satisfies \(\Majp(G)\le 3\).
\end{conjecture}

The bound, if true, is best possible as \(\Majp(C_{2k+1})= 3\). Note also that traversing an Euler tour and coloring its edges by the periodic pattern \(1,2,3,1,2,3,\dots\) yields a strong majority coloring whenever the length of the tour is divisible by three. Thus conjecture holds whenever $3\mid |E(G)|$.

\section{Constructions requiring four colors}

We first record three simple constructions showing that three colors do not suffice in
general.

\begin{construction}[Subdivided snarks~\cite{KKPW}]
Let \(S\) be a \emph{snark}, that is, a 3-regular graph with chromatic index \(4\), and let \(S^\ast\) be obtained from \(S\) by subdividing every edge exactly once. In other words, for every edge \(uv\in E(S)\) we add to $S^*$ a new vertex $x$, and we replace the original edge $uv$ with two edges $ux$ and $xv$.

We claim that \(S^\ast\) has no strong majority edge-coloring with three colors. Indeed, suppose such a coloring $c:E(S^\ast)\to[3]$ exists. Let \(uxv\) be the path replacing an edge \(uv\in E(S)\). Since every edge of \(S^\ast\) is adjacent to exactly three edges, the three edges adjacent to any fixed edge must receive pairwise distinct colors. It follows first that the three edges incident with any original vertex of \(S\) have pairwise distinct colors. Hence, for the edge \(ux\), the two other edges incident with \(u\) already use the two colors different from \(c(ux)\). Since the three edges adjacent to \(ux\) must be pairwise distinct, we must have $c(xv)=c(ux)$. Thus, the two edges of every subdivided edge receive the same color. Contracting each subdivided path gives a~proper edge-coloring of \(S\) with three colors, a contradiction. Hence \(S^\ast\) requires at least four colors for any strong majority edge-coloring.
\end{construction}

\begin{construction}[A pendant star \(K_{1,4}\)]
Let \(G\) be any admissible graph and let \(v_0\in V(G)\). Add a~new vertex \(x\), join \(x\) to \(v_0\), and add three new leaves \(v_1,v_2,v_3\) adjacent to \(x\). Thus \(x\) is the center of a pendant star \(K_{1,4}\).

Note that the four edges incident with \(x\) must receive pairwise distinct colors in every
strong majority edge-coloring. Indeed, if two of them have the same color, say $c(xv_i)=c(xv_j)$, $i\neq j$, then the edge \(xv_k\), $k\notin\{i,j\}$, is adjacent to exactly three edges, two of which have the same color. Therefore, every graph containing such a pendant \(K_{1,4}\) requires at least four colors in the strong majority edge-coloring.
\end{construction}

\begin{construction}[Fig.~\ref{fig.1}]\label{cons:3}
Let \(H\) be the graph obtained from the path $v_1v_2v_3v_4v_5v_6$ by adding the
edge \(v_3v_5\). We call \(v_1v_2\) and \(v_5v_6\) the terminal edges of \(H\).
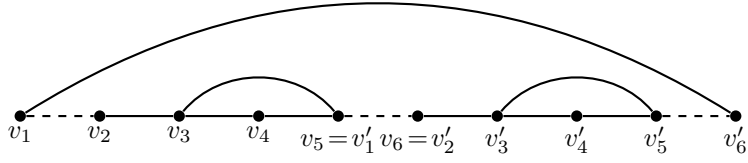
\begin{figure}
\begin{center}
\begin{tikzpicture}[scale=0.75,every node/.style={font=\small},
                    dot/.style={circle,fill,inner sep=1.5pt}]
\node[dot] (v1)  at ( 0.0,0) {};  \node[below] at (v1)  {$v_1$};
\node[dot] (v2)  at ( 1.4,0) {};  \node[below] at (v2)  {$v_2$};
\node[dot] (v3)  at ( 2.8,0) {};  \node[below] at (v3)  {$v_3$};
\node[dot] (v4)  at ( 4.2,0) {};  \node[below] at (v4)  {$v_4$};
\node[dot] (v5)  at ( 5.6,0) {};  \node[below] at (v5)  {$v_5\!=\!v_1'$};
\node[dot] (v6)  at ( 7.0,0) {};  \node[below] at (v6)  {$v_6\!=\!v_2'$};
\node[dot] (v3p) at ( 8.4,0) {};  \node[below] at (v3p) {$v_3'$};
\node[dot] (v4p) at ( 9.8,0) {};  \node[below] at (v4p) {$v_4'$};
\node[dot] (v5p) at (11.2,0) {};  \node[below] at (v5p) {$v_5'$};
\node[dot] (v6p) at (12.6,0) {};  \node[below] at (v6p) {$v_6'$};
\draw[thick] (v2)--(v3)--(v4)--(v5);
\draw[thick] (v6)--(v3p)--(v4p)--(v5p);
\draw[thick] (v3) to[bend left=50] (v5);
\draw[thick] (v3p) to[bend left=50] (v5p);
\draw[thick,dashed] (v1)--(v2);
\draw[thick,dashed] (v5)--(v6);
\draw[thick,dashed] (v5p)--(v6p);
\draw[thick] (v1) to[bend left=32] (v6p);
\end{tikzpicture}
\caption{Two copies of the graph \(H\) from Construction 3 glued at the shared terminal edge \(v_5v_6=v_1'v_2'\) and equipped with additional edge \(v_1v_6'\). In any strong majority \(3\)-edge-coloring, the three dashed terminal edges receive the
same color; as they are the only two edges adjacent to \(v_1v_6'\), the strong majority condition fails.}\label{fig.1}
\end{center}
\end{figure}

If there exists a strong majority edge-coloring of \(H\) with three colors, then the two terminal edges must have the same color.  Indeed, the edge \(v_2v_3\) is adjacent to $v_1v_2,~v_3v_4,~v_3v_5$, and these three edges must have pairwise distinct colors. Similarly, the edge \(v_4v_5\) is adjacent to $v_3v_4,~v_3v_5,~ v_5v_6$, which also must have pairwise distinct colors.  Hence, both \(v_1v_2\) and \(v_5v_6\) are forced to be colored with a unique color different from the colors of \(v_3v_4\) and \(v_3v_5\) (see~Fig.~\ref{fig.1}).

Now take two copies of \(H\), one on the vertex set $\{v_1,v_2,\dots,v_6\}$, the other on the vertex set $\{v_1',v_2',\dots,v_6'\}$, identify their terminal edges $v_5v_6$ with $v_1'v_2'$, and then add an edge joining the two remaining end vertices, that is the edge $v_1v_6'$. In any three-color strong majority edge-coloring, the last two terminal edges are forced to have the same color. But they are precisely the two edges adjacent to  $v_1v_6'$, whose edge-neighborhood has size two. This violates the strong majority condition.

By inserting more copies of the same gadget in a chain, this construction gives arbitrarily large admissible graphs of maximum degree three that require four colors for a strong majority edge-coloring.

\end{construction}

\section{Proof of Theorem~\ref{thm:five}}\label{sec:five}

Throughout this section, let $G$ be an admissible graph. 
Since cycles satisfy $\Majp(C_n)\le 3$, we may assume that $G$ is connected and is not a cycle. 

We first reduce $G$ to a graph with all vertex degrees not divisible by~$5$ via a sequence of almost degree-preserving operations. We also take care of the edges lying on paths consisting of degree two vertices.
We then apply the equitable edge-coloring Theorem~\ref{thm:HdW} of Hilton and de~Werra, for five colors. Finally, we pull the coloring back through the reductions, locally recoloring along the degree-two paths so that the strong majority condition holds everywhere.

For an edge-coloring $c:E(G)\to[k]$ of a graph $G$, a vertex $v$ and a color $i\in[k]$, let $d_i(v)$ denote the number of edges of color $i$ incident with $v$ in $G$. We say that the coloring $c$ is \emph{balanced} at $v$ if
\begin{equation}\label{eq:balanced}
        d_i(v)\;\le\;\frac{d_G(v)-1}{2}
        \qquad\text{for every color }i \in [k].
\end{equation}

\begin{remark}\label{rem:balanced-suffices}
If \eqref{eq:balanced} holds at both endpoints of an edge $e=uv$, then $e$ is strong majority colored: for every color $i$ the number of $i$-colored edges adjacent to $e$ is at most $d_i(u)+d_i(v)\;\le\;\frac{d_G(u)+d_G(v)}{2}-1.$
The same conclusion holds when one endpoint, say $u$, has degree $1$, since admissibility forces $d_G(v)\ge 3$ and the bound $d_i(v)\le \left(d_G(v)-1\right)/2\le \left(d_G(u)+d_G(v)\right)/2-1$ already suffices. Consequently, once the coloring is balanced at every vertex of degree at least three, the only edges that can fail the strong majority condition are those incident with a vertex of degree two.
\end{remark}

We record the elementary counting lemma that governs how much freedom we have when recoloring an edge at a vertex of degree at least three. It is stated for the five-color palette $[5]=\{1,2,3,4,5\}$.

\begin{claim}\label{claim:bad}
Let $d_G(v)=k\ge 3$ and suppose that \eqref{eq:balanced} holds at $v$. Call a color \emph{bad} for an edge $vw$ incident with $v$ if recoloring $vw$ with that color would violate \eqref{eq:balanced} at $v$. Then $vw$ has at most three bad colors if $k=4$, and at most two bad colors otherwise.
\end{claim}

\begin{proof}
A color is bad precisely when it already occurs $\lfloor (k-1)/2\rfloor$ times among the $k-1$ edges incident with $v$ and other than $vw$. For $k=3$ and $k=4$ we have $\lfloor (k-1)/2\rfloor=1$, and by \eqref{eq:balanced} each edge at $v$ carries a distinct color, so the bad colors are exactly the colors of the remaining $k-1$ edges; this gives two bad colors when $k=3$ and three when $k=4$. For $k\ge 5$ we have $\lfloor (k-1)/2\rfloor\ge 2$ and
$3\lfloor (k-1)/2\rfloor\geq 2\lfloor (k-1)/2\rfloor+2>  k-1$, so at most two colors can occur with the required multiplicity, hence at most two bad colors.
\end{proof}

A path $B=v\,e_0\,w_1\,e_1\cdots w_k\,e_k\,u$ in $G$ with $k\ge 1$ is an \emph{ear} if $d_G(w_1)=\dots=d_G(w_k)=2$ and $d_G(v),d_G(u)\ge 3$; its \emph{ends} are $v$ and $u$ (possibly with $v=u$, in which case $B$ is a cycle through $v$, called a \emph{loop ear}) and its \emph{length} is the number of its edges, that is, $k+1$. The
\emph{terminal edges} of $B$ are $e_0=vw_1$ and $e_k=w_ku$. Since $G$ is not a cycle and is connected, every vertex of degree two lies in the interior of a unique ear, and ears, together with the edges spanned by vertices of degree at least three, partition the set $E(G)$.

We transform $G$ into a graph $G_5$ by the following five operations, each applied exhaustively in the order listed below. Here, $G_j$ denotes the graph obtained after the $j$-th operation. Every operation preserves admissibility and leaves the degree
of every vertex of degree at least three that is not divisible by $5$, unchanged.

\medskip
\noindent\textbf{(R1) Loop ears.}
Let $B=v\,e_0\,w_1\,e_1\cdots w_k\,e_k\,v$ be a loop ear at vertex $v$. Delete all its vertices except $v$. If $d_G(v)\neq 5$, attach two new leaves at $v$, restoring the contribution of $2$ that $B$ made to $d_G(v)$, so that $d_G(v)$ is unchanged; if $d_G(v)=5$,
attach a single leaf to $v$, lowering $d_G(v)$ from $5$ to $4$. The operation preserves admissibility, and we record the correspondence between $B$ and its replacement for the reverse pass. Let $G_1$ be the resulting graph, and note that it has no loop ears.

\medskip
\noindent\textbf{(R2) Ears at a degree-five end.}
Let $B=v\,e_0\,w_1\,e_1\cdots w_k\,e_k\,u$ be an ear with $d_{G_1}(v)=5$ and $d_{G_1}(u)\ge 3$. Delete all non-terminal vertices of $B$, that is, $w_1,\dots,w_k$, and their incident edges, lowering $d_{G_1}(v)$ from $5$ to $4$ and $d_{G_1}(u)$ by one. Then, if $d_{G_1}(u)\neq 5$, attach an auxiliary leaf $w'$ at $u$, restoring $d_{G_1}(u)$ to its original value. The operation preserves admissibility and creates no new vertex of degree five. We record the correspondence between $B$ and its replacement for the reverse pass. Denote the resulting graph by $G_2$, and note that there are no ears with a degree-five end in it.

\medskip
\noindent\textbf{(R3) Length-two ears (cherries).}
Let $B=v\,e_0\,w_1\,e_1\,u$ be a cherry with $d_{G_2}(v),d_{G_2}(u)\ge 3$. If $vu\notin E(G_2)$, delete $w_1$ and add the edge $vu$; if $vu\in E(G_2)$, leave $B$ unchanged. The resulting graph is $G_3$. We record, for each contracted cherry, the pair of terminal edges it is to be expanded into.

\medskip
\noindent\textbf{(R4) Long ears.}
Let $B=v\,e_0\,w_1\,e_1\cdots w_k\,e_k\,u$ with $k\ge 2$ and $d_{G_3}(v),d_{G_3}(u)\ge 3$. Delete its interior and replace it by a single degree-two vertex joined to $v$ and to $u$, i.e.\ by a cherry $v\,w_1'\,u$. The resulting graph is $G_4$. We record the correspondence between $B$ and the cherry that replaced it.

\medskip
\noindent\textbf{(R5) Degrees divisible by five.}
For every vertex $x$ with $5\mid d_{G_4}(x)$, attach one new leaf at~$x$. The resulting graph is $G_5$.

\medskip
By construction, for every vertex $x\in V(G_5)$ we have $5\nmid d_{G_5}(x)$. Indeed, operations 
(R1)--(R2) turn degree five vertices into degree four, operations (R3)--(R4) do not change the degrees of vertices, and operation (R5) turns each degree divisible by five into a degree congruent to $1$ modulo five.

We may now apply Theorem~\ref{thm:HdW} with $k=5$ to the graph $G_5$, to get an equitable edge-coloring $c_5\colon E(G_5)\to[5]$. 
Deleting the leaves added in~(R5) yields an edge-coloring $c_4\colon E(G_4)\to[5]$.

\begin{claim}\label{claim:balanced-G4}
The coloring $c_4$ is balanced at every vertex of $G_4$ of degree at least three.
\end{claim}

\begin{proof}
Fix $x\in V(G_4)$ with $d_{G_4}(x)=d\ge 3$, and let $d'=d_{G_5}(x)$, so $d'=d+1$ if $5\mid d$ and $d'=d$ otherwise; in either case $5\nmid d'$. Equitability of the coloring $c_5$ and the deletion of at most one incident leaf yields that for any $i\in[5]$, the number of edges in color $i$, incident with $x$ in $G_4$, is at most $\lceil d'/5\rceil$, and the assertion follows since $\lceil d'/5\rceil\le(d-1)/2$ holds for every $d\ge 3$.
\end{proof}

By Remark~\ref{rem:balanced-suffices}, every edge of $G_4$ whose endpoints both have degree at least three is already strong majority colored under $c_4$. We record the one feature of the reduction that drives the restoration of the degree-five vertices.

\begin{remark}\label{claim:free-color}
Let $v$ be a vertex of $G$ with $d_G(v)=5$ processed by (R1) or (R2). Exactly one of these operations is applied at $v$ and it removes at most two loop ears in the case of (R1) or a single ear incident with $v$ in the case of (R2). Consequently, $d_{G_4}(v)=4$ and the four edges at $v$ carry four distinct colors under $c_4$.
\end{remark}

It remains to pull $c_4$ back to $E(G)$ and to adjust the colors on the edges incident with vertices of degree two. The mechanism is the following local statement.

\begin{claim}\label{claim:extend}
Let $H$ be a graph consisting of
\begin{itemize}
    \item a path $B=v\,e_0\,w_1\,e_1\cdots w_k\,e_k\,u$, where $k\ge 1$, $v\neq u$, and $uv\in E(G)$ whenever $k=1$,
    \item at least two edges incident with $v$ and different from $e_0$ (possibly with $uv$),
    \item at least two edges incident with $u$ and different from $e_k$ (possibly with $uv$).
\end{itemize} 
Suppose $d_H(v)\neq 5$ and $d_H(u)\neq 5$, and we have a coloring $c: E(H)\to [5]$, which is balanced at both $v$ and $u$, and which satisfies $c(e_0)\neq c(e_k)$ in case $k\geq 2$. Then $e_0,\dots,e_k$ can be recolored so that every edge of $B$ is strong majority colored and the new coloring is still balanced at both $v$ and $u$.
\end{claim}

\begin{proof}
We consider two cases.

\smallskip
\emph{Case $k=1$}
\smallskip

Note that we have $uv\in E(H)$. If $e_0$ and $e_1$ are already strong majority colored, we are done. If $d(v)\ge 4$ then the balanced condition at $v$ makes the edge $e_0$ strong majority colored, and symmetrically $e_1$ if $d(u)\ge 4$. Hence, an adjustment is needed only for the case $d_H(v)=3$ or $d_H(u)=3$. Suppose $d_H(v)=3$ and $e_0$ is not strong majority colored. We will recolor $e_1$. By Claim~\ref{claim:bad}, at most three colors are forbidden at $u$, and since $uv\in E(H)$, the color $c(uv)$ is one of them if there are exactly three colors forbidden. For the strong majority condition for the edge $e_0$, there are two colors forbidden, and one of them is $c(uv)$ again. Hence, a color remains for $e_1$ that keeps the balance at $u$ and makes $e_0$ strong majority colored. The symmetric failure at $u$ can be repaired in the same way by recoloring $e_0$. Note that by this procedure, no edge other than $e_0$ or $e_1$ will change its color.

Note that we have $uv\in E(H)$. If $e_0$ and $e_1$ are already strong majority colored, we are done. So suppose $e_0$ is not strong majority colored. We will adjust the color of $e_1$, preserving the balance at $u$. By Claim~\ref{claim:bad}, at most three colors are forbidden at $u$, and since $uv\in E(H)$, the color $c(uv)$ is one of them if there are exactly three colors forbidden. For the strong majority condition for the edge $e_0$, there are at most two colors forbidden for $e_1$, and in case of exactly two forbidden colors, one of them is $c(uv)$ again. Hence, a color remains for $e_1$ that keeps the balance at $u$ and makes $e_0$ strong majority colored. The symmetric failure at $e_1$ can be repaired in the same way by recoloring $e_0$. Note that by this procedure, no edge other than $e_0$ or $e_1$ will change its color.

\smallskip
\emph{Case $k\ge 2$.}
\smallskip

When $k=2$, the single interior edge $e_1$ is adjacent only to $e_0$ and $e_2$, which already have distinct colors, so $e_1$ satisfies the strong majority condition, regardless of its color. Now, since the coloring is balanced at $v$, there are at most two colors that, after being used to color $e_1$, could violate the strong majority condition for $e_0$, and the same holds for $e_2$. Therefore, there is a color available for $e_1$, i.e.~a color that makes both $e_0$ and $e_2$ strong majority colored. 

If $k>2$, we color the edges $e_j$, $j\in[k-1]$, one-by-one. Note that any such edge joins two degree-two vertices. Hence, while coloring $e_j$, we are restricted only by the strong majority condition for edges $e_{j-1}$ and $e_{j+1}$, and from each side there are no more than two forbidden colors. A free color for $e_j$ thus remains, and the coloring of the interior extends greedily. 
\end{proof}

\subsection*{Pulling the coloring back to $G$}

We undo the operations in reverse order, that is $G_4\to G_3\to G_2\to G_1\to G$, extending the coloring at each step.

\smallskip
\emph{From $c_4$ to $c_3$ (undoing (R4)).}
\smallskip

Each long ear of $G_3$ was replaced by a cherry $v\,w_1'\,u$ in $G_4$, whose two edges carry distinct colors $c_4(vw_1')\neq c_4(uw_1')$, since the initial coloring was equitable. Restore the ear, set $c_3(e_0)=c_4(vw_1')$ and $c_3(e_k)=c_4(uw_1')$,  this leaves the color multiset at $v$ and at $u$ unchanged, so the balanced condition persists there. Next, color the interior greedily as in the case $k\ge 2$ of Claim~\ref{claim:extend}. The remaining edges preserve their colors. This yields $c_3\colon E(G_3)\to[5]$.

\smallskip
\emph{From $c_3$ to $c_2$ (undoing (R3)).}
\smallskip

Restore each contracted cherry $v\,w_1\,u$ and note that no end vertex has degree five here. If the cherry had been replaced by the edge $uv$, the two restored edges inherit the color of $uv$, that is we set $c_2(e_0)=c_2(e_1)=c_3(uv)$. Then the balance condition at $v$ is preserved and, thanks to that, $e_0$ is strong majority colored. The same holds for $u$ and $e_1$. If the cherry had been retained because $uv\in E(G)$, recolor $e_0, e_1$ by the case $k=1$ of Claim~\ref{claim:extend}. The remaining edges preserve their colors. This yields $c_2\colon E(G_2)\to[5]$.

\smallskip
\emph{From $c_2$ to $c_1$ (undoing (R2)).}
\smallskip

Let $B=v\,e_0\,w_1\,e_1\cdots w_k\,e_k\,u$ be an ear removed by~(R2), with $d_G(v)=5$ and $v\neq u$, and restore its interior. By Remark~\ref{claim:free-color}, the four edges at $v$ in $G_2$ carry distinct colors, so $v$ does not restrict $e_0$: any color leaves $v$ balanced, raising at most one color to multiplicity $2=\lfloor(d_G(v)-1)/2\rfloor$, and $e_0$ is strong majority colored for the same reason. If $d_G(u)\neq 5$, delete the auxiliary edge $uw'$ attached by~(R2) and set $c_1(e_k)=c_2(uw')$. If $d_G(u)=5$, the end $u$ is unrestricted in the same way as $v$. If $k\geq 2$, we can now color all remaining edges of $B$ greedily, starting from $e_1$ to $e_{k-1}$, as in case $k\geq 2$ of Claim~\ref{claim:extend}. The remaining edges preserve their colors. This yields a coloring $c_1\colon E(G_1)\to[5]$.

\smallskip
\emph{From $c_1$ to $c$ (undoing (R1)).}
\smallskip

For a loop ear with $d_G(v)\neq 5$, deleting the two auxiliary leaves attached at $v$ frees their colors for the terminal edges $e_0,e_k$, and the interior extends greedily as before; the only obstruction is when $k=2$ and both $e_0$ and $e_k$ inherit the same color. However, this forces $d_G(v)\ge 6$, since balance assigns distinct colors to all edges at a vertex of degree at most four. Then by Claim~\ref{claim:bad}, at most two colors are bad at $v$, so we may recolor $e_k$ with a color distinct from that used for $e_0$ and from the at most two bad colors at $v$. This keeps the balance at $v$, separates the colors of $e_0$ and $e_k$, and lets $e_1$ be strong majority colored. We can then find a color for the edge $e_1$, using the same strategy as in case $k\geq 2$ of Claim~\ref{claim:extend}. 

For $d_G(v)=5$, the single leaf was removed, so one terminal edge inherits its color, and the other is free: any choice keeps the balance at $v$, and the loop interior then extends as before.

Every edge of $G$ is therefore strongly majority colored, whence $\Majp(G)\le 5$. This completes the proof. \qed

\section*{Acknowledgments}
The research of S.~Antoniuk was supported by the National Science Centre, grant 2024/53/B/ ST1/00164.
The research of M.~Prorok and N.~Salia was supported by the National Science Centre, grant 2021/42/E/ST1/00193.

The authors used GPT-5.5 for language editing and expository
suggestions. The authors are solely responsible for the mathematical content and the
final manuscript.

\bibliographystyle{abbrv}
\bibliography{references.bib}

\end{document}